\documentclass[10pt]{amsart}
\usepackage{amssymb, amscd, amsthm, epsfig} 
\newtheorem{thm}{\bfseries Theorem}
\newtheorem{cor}[thm]{\bfseries Corollary}

\newtheorem{conj}[thm]{\bfseries Conjecture}

\def\R{\mathbb R}
\def\F{\mathbb F}

\def\per{\mbox{\rm Per}}
\def\sgn{\mbox{\rm sgn}}

\def\Det{\mbox{\rm Det}}
\def\rk{\mbox{\rm rk}}
\def\ac{\mbox{\rm ac}}

\def\S{{\mathrm S}}

\begin{document}


\title[\bf Covering the permutohedron by hyperplanes \hfill {}]
{\bf Covering the permutohedron \\ 
by affine hyperplanes}


\author[ {} \hfill Gy. K\'arolyi, G. Heged\"us]{}

\address{\vskip -.1in}


\email{ karolyi.gyula@renyi.hu ``Gyula K\'arolyi''}
\email{ hegedus.gabor@uni-obuda.hu ``G\'abor Heged\"us''}



\maketitle
\vskip 1in


\noindent GYULA K\'AROLYI\hskip.1in
\noindent HUN--REN Alfr\'ed R\'enyi Institute of Mathematics,
Re\'al\-ta\-noda utca 13--15, H--1053 Budapest, Hungary
   and
Department of Algebra and Number Theory, 
E\"otv\"os University, 
P\'azm\'any P. s\'et\'any 1/C,
Budapest, H--1117 Hungary 
\smallskip

\noindent G\'ABOR HEGED\"US\hskip.1in
\noindent Institute of Applied Mathematics,
John von Neumann Faculty of Informatics,
\'Obuda University,
B\'ecsi \'ut 96/B,
Budapest, H--1032 Hungary

\vskip.85in

\setlength{\baselineskip}{15pt}


\noindent\textsc{Abstract.}
An almost cover of a finite set in the affine space is a collection of
hyperplanes that together cover all points of the set except one. Using
the polynomial method, we determine the minimum size of an almost cover of
the vertex set of the permutohedron and address a few related questions.
\vskip.2in



\bigskip
\section{Introduction} 

\noindent
Consider those points in the $n$-dimensional space whose coordinates
form a permutation of the first $n$ positive integers.
The elements of this set $P_n$ are the vertices
of a convex $(n-1)$-dimensional polytope called the permutohedron
(spelled also as permutahedron) $\Pi_{n-1}$.
For $n=3$ it is a regular hexagon, for $n=4$ a truncated octahedron. This
polytope has many fascinating properties and can be used to illustrate
various concepts in geometry, combinatorics and group theory.
For example, it is a simple polytope which tiles the affine hyperplane it
is living in. It is also a zonotope: a Minkowski sum of
$n\choose 2$ line segments.
It has $2^n-2$ facets, corresponding to
the nonempty proper subsets of $[n]$. In general, the number of faces of 
fixed dimension relates to Stirling numbers of the 2nd kind.
The symmetric group $\S_n$ acts on $\Pi_{n-1}$ by permuting the coordinates
and can be viewed as a subgroup of the group of its isometries. 
Using the natural labeling of the vertices
by the elements of $\S_n$ one can identify its edge skeleton
with the undirected Cayley graph of $\S_n$
defined by $n-1$ transpositions. An alternative labeling of the vertices
obtained by the anti-isomorphism $\pi \to \pi^{-1}$
reflects the construction of $\Pi_{n-1}$ as the
monotone path polytope of the $n$-cube, see \cite{Z}. 
\medskip

\noindent
An almost cover of a finite set in the affine space is a collection of
hyperplanes that together cover all points of the set except one. According to
a classical result of Jamison \cite{J}, an almost cover of the
$n$-dimensional affine space over the $q$-element finite field requires at
least $(q-1)n$ hyperplanes. Equivalently, to pierce every affine hyperplane in
$\F_q^n$ one needs at least $(q-1)n+1$ points, see \cite{BS2}. See also
\cite{BBS} for further results in finite geometries.
Another example is the
Alon--F\"uredi theorem \cite{AF}: {\em Every almost cover of the vertex set
of an $n$-dimensional cube requires at least $n$ hyperplanes.}
The key observation of the present note is the following analogue.

\begin{thm}
Every almost cover of the vertices of $\Pi_{n-1}$
consists of at least $n\choose 2$ hyperplanes. This bound is sharp.
\label{main}
\end{thm}

\noindent
In comparison, all vertices of $\Pi_{n-1}$ are contained in just
one hyperplane, and even
when $\Pi_{n-1}$ is embedded into $\R^{n-1}$, its vertices
can be covered by $n$ hyperplanes, see Section 3.
Our proof of the above theorem relying on the polynomial
method parallels that of the Alon--F\"uredi theorem.
After recalling the Combinatorial Nullstellensatz, it is given in Section 3.
In Section 4 we prove a similar statement for prisms over permutohedra. 
In Section 5 we extract a version of the polynomial lemma that may be useful
in the study of permutations.
We conclude this paper by suggesting a common generalization of
Theorem \ref{main} and the Alon--F\"uredi theorem to zonotopes.

\bigskip
\section{Preliminaries}

\noindent
The {Combinatorial
Nullstellensatz}, formulated by Noga Alon in the late nineties,
describes, in an efficient way, the structure of multivariate polynomials 
whose zero-set includes a Cartesian product over a field $\F$.
This characterization immediately implies the following non-vanishing
criterion \cite{A1}, based on which one can prove Theorem \ref{main} via
Theorem \ref{perm} derived in Section 5.

\begin{thm} 
Let $f=f(x_1,\ldots ,x_n)$ be a polynomial in $\F[x_1,\ldots, x_n]$.
Suppose that there is a monomial $\prod_{i=1}^n x_i^{d_i}$  such that 
$\sum_{i=1}^n d_i$ equals the degree of $f$ and whose coefficient in $f$
is nonzero. If  $S_1,\ldots ,S_n$ are subsets of $\F$ with
$|S_i|>d_i$, then there are $s_1\in S_1, \ldots, s_n\in S_n$
such that $f(s_1,\ldots,s_n)\neq 0$.
\label{poly}
\end{thm}

\noindent
This result has innumerable variations with even more different proofs,
see e.g. \cite{R}. Apparently they all depend on two basic principles:
reduction modulo a standard Gr\"obner basis and Lagrange interpolation.
Particularly useful for us is the following variant.

\begin{thm}
Let $S_1,\ldots ,S_n$ be subsets of $\F$, $|S_i|=k_i$, and let $f$ be a
polynomial in $\F[x]=\F[x_1,\ldots, x_n]$ whose degree is at most
$\sum_{i=1}^n (k_i-1)$.
\begin{enumerate}

\item[(i)] If $f({s})=0$ for every
  ${s}\in S_1\times\dots\times S_n$, then the coefficient of
  the monomial $\prod_{i=1}^n x_i^{k_i-1}$ in $f$ is zero.

\item[(ii)] If $f({s})=0$ for all but one element
  ${s}\in S_1\times\dots\times S_n$, then the coefficient of
  the monomial $\prod_{i=1}^n x_i^{k_i-1}$ in $f$ is not zero.

\end{enumerate}
\label{punctured}
\end{thm}

\noindent
While the first statement is just the initial formulation of Theorem \ref{poly},
the second one already points towards applications for almost covers.
It is a very special case of Corollary 4.2 in \cite{BS1} and can be derived 
directly from (i) rather easily.
It also implies the following version of Theorem 5 in \cite{AF},
which can be used for the most transparent direct proof of Theorem \ref{main}.

\begin{thm} 
Let $S_1,\dots,S_n$ be nonempty subsets of $\F$, $B=S_1\times\dots\times S_n$.
If a polynomial $f\in \F[x_1,\ldots, x_n]$ vanishes at every point of $B$
except one, then its degree is at least $\sum_{i=1}^n(|S_i|-1)$.
\label{AF}
\end{thm}

\bigskip
\section{Covers of the permutohedron}

\noindent
All points of $P_n$ lie in the hyperplane $H_0$ of equation
$f_0(x)=\sum_{i=1}^nx_i-{{n+1}\choose 2}=0$. However, if we view $\Pi_{n-1}$ as an
$(n-1)$-dimensional polytope, the meaningful 
question is the following:
{\em How many affine hyperplanes different from $H_0$ are needed to cover
  all points of $P_n$?}
Surprisingly this appears to be a more elusive problem.
It is clear that the hyperplanes of equation $x_1=i$,
$1\le i\le n$ cover $P_n$. Each of them contains exactly $(n-1)!$ points
of $P_n$, and their intersections with $P_n$ are pairwise disjoint.
The same is true for the hyperplanes of equation $x_i=1$.
If $n\ge 4$ is even, then one can do even better: The hyperplanes of equation
$x_1+x_j=n+1$,
$2\le j\le n$ cover $P_n$. Each of them contains exactly $n(n-2)!$ points
of $P_n$, and their intersections with $P_n$ are pairwise disjoint.
We believe that these examples are extremal in the following sense.
\medskip

\begin{conj}
Suppose that the vertex set of $\Pi_{n-1}$ is contained in the union
of the hyperplanes $H_1,\ldots,H_m$ different from $H_0$.
If $n$ is odd, then $m\ge n$. If $n\ge 4$ is even, then $m\ge n-1$.
\label{cover}
\end{conj}

\noindent
Consider a hyperplane $H$ not parallel to $H_0$, it intersects $H_0$ in
a 1-codimensional affine subspace. To cover $P_n$ by $n$ or less such
hyperplanes one needs to find such an $H$ which intersects $P_n$ in at 
least $(n-1)!$ points.
$H$ has an equation of the form $f(x)=0$,
where $f$ is a linear polynomial, in this case we write $H=H(f)$. Then
$H(g)\cap H_0=H(f)\cap H_0$ if and only if there exist $\alpha,\beta\in\R$,
$\alpha\ne 0$ such that $g=\alpha f+\beta f_0$.
Apart from such equivalence, it seems that $H(f)$ intersects $P_n$ in more
than $(n-1)!$ points if and only if $n$ is even and $f=x_i+x_j-(n+1)$
for some $i\ne j$. $|H(f)\cap P_n|=(n-1)!$ occurs
in each dimension for $f=x_i-x_j-1$,
and also for $f=x_i+x_j-n$, $f=x_i+x_j-(n+1)$, $f=x_i+x_j-(n+2)$ when $n$
is odd. From these examples one can construct various
economical hyperplane covers of $P_n$. For example, the hyperplanes
of equation $x_n=1$, $x_n-x_i=1$, $1\le i\le n-1$ cover $P_n$ for every $n$,
whereas the hyperplanes of equation $x_1=(n+1)/2$, $x_1+x_j=n+1$,
$2\le j\le n$ cover $P_n$ when $n$ is odd.
We believe that the following is true.

\begin{conj}
If $n$ is odd, then every hyperplane different from $H_0$ contains
at most $(n-1)!$ points of $P_n$.
\label{(n-1)!}
\end{conj}

\noindent
If true, this would imply Conjecture \ref{cover} as follows. The statement
follows immediately if $n$ is odd. Suppose that
$n\ge 4$ is even and the hyperplanes $H_1,\dots,H_m$ cover $P_n$.
We may assume that $m\le n-1$.
Consider the hyperplanes $H(x_j-n)$ for $1\le j\le n$. They intersect
$\Pi_{n-1}$ in pairwise disjoint respective facets $F_j$,
congruent to $\Pi_{n-2}$ apiece,
of which no two can be contained in the same hyperplane $H_i$. By the
pigeonhole principle there is a facet
$F_j$, which is not fully contained in any $H_i$, $1\le i\le m$.
Within the $(n-1)$-dimensional ambient space $H_0$, consider the $k\le m$
hyperplanes $H_i\cap H_0$ which intersect the affine hull of $F_j$.
They cover all the vertices of $F_j$ and are different from its affine hull
$H(x_j-1)\cap H_0$. Applying a suitable isometry, it follows from
Conjecture \ref{(n-1)!} that each of them contains  at most $(n-2)!$
vertices of $F_j$. Accordingly, $k(n-2)!\ge (n-1)!$, yielding
$m\ge k\ge n-1$.
\medskip

\noindent
Turning to almost covers of the vertex set of $\Pi_{n-1}$, note that unlike
to the case of exact covers it
does not matter which ambient space we consider. Throughout this paper it will
be convenient to use the notation
$$V({x})=\prod_{1\le i<j\le n}(x_j-x_i)$$
for the Vandermonde polynomial in the variable $x=(x_1,\dots,x_n)$.
\medskip

\noindent{\em Proof of Theorem \ref{main}.}
A key observation, due to Alon \cite{A2}, which leads to a by now
standard application of the Combinatorial Nullstellensatz is that the
permutations of an $n$-element set $A$ can be identified with the set
obtained from $A^n$ by removing
the zeroes of the Vandermonde polynomial $V({x})$.
Geometrically, $P_n$ is obtained from $\{1,2,\ldots,n\}^n$ by removing the
points covered by the $n\choose 2$ hyperplanes of equation $x_i=x_j$.

Suppose that a collection of hyperplanes defined by the
linear equations $f_1({x})=0,\dots,f_m({x})=0$ constitute an almost
cover of $P_n$ where the uncovered vertex is ${v}$. Then the polynomial
$$f(x)=V({x})\prod_{i=1}^m f_i({x})\in \R[x_1,\ldots,x_n]$$
of degree ${n\choose 2}+m$
attains the value 0 at every point of the Cartesian product
$\{1,2,\dots,n\}^n$ except ${v}$. 
According to Theorem \ref{AF},
$${n\choose 2}+m\ge n(|A|-1)=n(n-1),$$
hence $m\ge {n\choose 2}$.
To see that this bound cannot be improved, notice that the hyperplanes
$x_i=j$ ($i<j$) cover every vertex but $(1,2,\ldots,n)$. Since $\S_n$ acts
transitively on $P_n$, there is a similar construction for each further
vertex of $\Pi_{n-1}$.
\hfill\qed
\medskip

\noindent{\em Remarks.} Replacing $\prod f_i$ by an arbitrary polynomial
of degree $m$ in the above proof one obtains the following more general
theorem: {\em If an algebraic
surface covers all vertices of $\Pi_{n-1}$ except one, then its degree is
at least $n\choose 2$.} The proof also works verbatim for sets of the form
$$X(\alpha_1,\ldots,\alpha_n)=
\{(\alpha_{\pi(1)},\ldots,\alpha_{\pi(n)}) \mid \pi\in \S_n\},$$
where $\alpha_1, \ldots , \alpha _n$ are different elements of an
arbitrary field $\F$.

\begin{cor}
Let $v,u_1,\dots,u_k\in X_n=X(\alpha_1,\ldots,\alpha_n)$ such that $v$ is
not contained in the affine hull of $\{ u_1,\dots,u_k\}$. 
Suppose that the affine
hyperplanes $H_1,\dots,H_m$ satisfy
$$X_n\setminus \{ u_1,\dots,u_k\}\subseteq H_1\cup\dots\cup H_m,\quad
v\not\in H_1\cup\dots\cup H_m.$$
Then $m\ge {n\choose 2}-1$.
\end{cor}

\begin{proof}
Choose an affine hyperplane $H$ which contains the affine hull of
$\{ u_1,\dots,u_k\}$ but does not contain $v$. Then $H,H_1,\dots,H_m$ 
constitute an almost cover of $X_n$, hence $m+1\ge {n\choose 2}$.
\end{proof}

\section{Prisms over permutohedra}

\noindent
Consider an $n$-dimensional prism whose bases are congruent to the
permutohedron $\Pi_{n-1}$. Pick any vertex $v$. Covering the vertices of the
base which does not contain $v$ by one hyperplane and applying the
construction given in the proof of Theorem \ref{main}
to the other base one obtains
an almost cover of the vertex set by ${n\choose 2}+1$ hyperplanes. 
Here we prove that this is best possible.
Because of affine invariance it is enough to
prove it for the prism whose bases are $\Pi_{n-1}$ and
$-\Pi_{n-1}=\Pi_{n-1}-(n+1)(e_1+\dots+e_n)$, where $e_1,\dots,e_n$ is the
standard orthonormal basis for $\R^n$.

\begin{thm}
Every almost cover of $P_n\cup(-P_n)$
consists of at least ${n\choose 2}+1$ hyperplanes. 
\label{prismperm}
\end{thm}

\begin{proof}
  Let $m={n\choose 2}$ and suppose that the hyperplanes $H_i$, $1\le i\le m$
  cover every point of $P_n\cup(-P_n)$ except $v$.
  By symmetry, we may assume that $v\in -P_n$. The hyperplane
  $H_i$ is defined by an equation $f_i({x})=a_i$ where $f_i$ is a linear
  form. Consider the Vandermonde polynomial $V({x})=\prod_{i<j}(x_j-x_i)$.
  The polynomial
  $$f({x})=V({x})\prod_{i=1}^m (f_i({x})-a_i))$$
  of degree $n(n-1)$ vanishes at every point of the Cartesian product
  $\{1,2,\dots,n\}^n$. By Theorem \ref{punctured} (i), the coefficient of the
  monomial $\prod_{i=1}^n x_i^{n-1}$ in $f$ must be zero.

  On the other hand, the polynomial $f$ attains the value 0 at every point of
  the Cartesian product $\{-1,-2,\dots,-n\}^n$ except ${v}$. That
  is, the polynomial  
  $$g({x})=f(-{x})=(-1)^{n\choose 2}V({x})\prod_{i=1}^m (-f_i({x})-a_i))
  =V({x})\prod_{i=1}^m (f_i({x})+a_i))$$
  of degree $n(n-1)$ vanishes at every point of the Cartesian product
  $\{1,2,\dots,n\}^n$ except $-{v}$.
  By Theorem \ref{punctured} (ii), the coefficient of the
  monomial $\prod_{i=1}^n x_i^{n-1}$ in $g$ must be nonzero.
  Since the degree $n(n-1)$ parts of the polynomials $f$ and $g$ are identical,
  we arrive at a contradiction.
\end{proof}

\noindent
For a far reaching generalization of this result we refer to \cite{K}.

\section{Permutations via polynomials}

\noindent
Besides Theorem \ref{main} and Snevily's problem treated in \cite{A2,DKSS},
there are several combinatorial problems regarding permutations that may
be treated with the polynomial method, see e.g. \cite{GHNP,KP}. Although
it is not needed for the present work, it still may be useful to formulate
the following variant of Theorem \ref{poly} for permutations, which can be
extracted from the proofs in \cite{A2,DKSS}.

\begin{thm} 
Let $\alpha_1, \ldots , \alpha _n$ be $n$ different elements of a field $\F$.
Let $f\in \F[x]$ be a polynomial of degree
${n\choose 2}$. For each permutation $\pi \in \S_n$, denote by $c_\pi$ the
coefficient of the monomial $\prod_{i=1}^n x_i^{\pi(i)-1}$ in $f$. If
$$\sum_{\pi\in \S_n} \sgn(\pi)c_{\pi}\ne 0,$$
then there exists a point $s\in X(\alpha_1,\ldots,\alpha_n)$
such that $f(s)\ne 0$.
\label{perm}
\end{thm}

\begin{proof}
Put $S_i=\{\alpha_1, \ldots , \alpha _n\}$ and consider the polynomial
$g(x)=V(x)f(x)$, then $\deg g=n(n-1)$. Using the determinant expression 
$$V(x)=\Det(x_i^{j-1})_{1\le i,j\le n}=
\sum_{\pi \in \S_n} \sgn(\pi) \prod_{i=1}^n x_i^{\pi(i)-1}$$
it is easy to compute the coefficient $c$ of the monomial
$\prod_{i=1}^nx_i^{n-1}$ in $g$. For each $\pi \in \S_n$ denote by $\pi'$
the permutation satisfying $\pi'(i)=n+1-\pi(i)$ for every $i$, then
$(\pi')'=\pi$ and $\sgn(\pi')=(-1)^{n\choose 2}\sgn(\pi)$. Thus,
$$c=\sum_{\pi\in \S_n} \sgn(\pi)c_{\pi'}=\sum_{\pi'\in \S_n} \sgn(\pi')c_{(\pi')'}=
(-1)^{n\choose 2}\sum_{\pi\in \S_n} \sgn(\pi)c_{\pi}\ne 0.$$
Since $g(x)=0$ for $x\in(S_1\times\dots\times S_n)\setminus
X(\alpha_1,\ldots,\alpha_n)$, the result follows from Thm \ref{poly}.
\end{proof}  

\noindent
As a quick application we derive Lemma 4 in \cite{DKSS}.
For an $n\times n$ matrix $M=(m_{ij})$, the permanent of $M$ is
$\per M=\sum_{\pi\in \S_n} m_{1\pi(1)}m_{2\pi(2)}\ldots m_{n\pi(n)}$.
Here $M=M(a_1,\ldots,a_n)$ will denote the Vandermonde matrix with entries
$m_{ij}=a_i^{j-1}$.

\begin{cor} 
Let $a_1,a_2,\ldots,a_n\in \F$ such that $\per M(a_1,\ldots,a_n)\ne 0$.
Then, for any subset 
$B\subset \F$ of cardinality $n$ there is a
numbering $b_1,\ldots ,b_n$ of the elements of $B$ such that the products
$a_1b_1,\ldots ,a_nb_n$ are pairwise different.
\label{mult}
\end{cor} 

\begin{proof}
Take $f(x)=V(a_1x_1,\dots, a_nx_n)=
\prod_{1\le i<j\le n}(a_jx_j-a_ix_i)$, then
$$f(x)=\sum_{\pi \in \S_n} \sgn(\pi) \prod_{i=1}^n (a_ix_i)^{\pi(i)-1}.$$
Thus, $c_\pi=\sgn(\pi)\prod_{i=1}^n a_i^{\pi(i)-1}$ and
$$\sum_{\pi\in \S_n} \sgn(\pi)c_{\pi}=
\sum_{\pi\in \S_n} (\sgn(\pi))^2\prod_{i=1}^n a_i^{\pi(i)-1}
=\per M(a_1,\ldots,a_n)\ne 0.$$
The claim follows applying Theorem \ref{perm} with
$\{\alpha_1, \ldots , \alpha _n\}=B$.
\end{proof}  

\section{Concluding remarks}

\noindent
For a finite set of points $X$ in an affine space and $v\in X$ define
$\ac(X,v)$ as the minimum number of hyperplanes whose union contains
$X\setminus \{v\}$ but does not contain $v$ and let
$$\ac(X)=\min_{v\in X}\ac(X,v).$$  
If $X$ is the vertex set of a convex polytope, we simply write $\ac(P)$.
Thus, $\ac(C_n)=n$, and $\ac(\Pi_{n-1})={n\choose 2}$.
For $n$-element sets $X\subset \R^d$, $\ac(X)$ attains every integer
between $1$ and $\lceil (n-1)/d \rceil$, even when $X$ is restricted
to sets in convex position. Thus one may expect a meaningful estimate on
$\ac(X)$ only if $X$ satisfies some appropriate geometric, combinatorial or
algebraic condition. Inspired by Theorem \ref{main} we found that
being the vertex set of a zonotope may be a reasonable indicator.
\medskip

\noindent
A zonotope is a convex polytope that can be represented as the Minkowski sum
of a finite number of line segments. A collection of line segments is called
nondegenerate if no two of the segments are parallel to each other. Each
zonotope $Z$ can be written as the Minkowski sum of a nondegenerate
collection of line segments.
Such a representation is unique up to translations of the segments by
vectors which add up to ${0}$. The number of the summands,
denoted by $\rk(Z)$, we call the rank of $Z$. 
\medskip

\noindent
Each zonotope of rank $n$ can be obtained as an orthogonal projection of an
affine image of the $n$-dimensional unit cube
$C_n$, see \cite{Z}. If the dimension of a zonotope is
1, then it is a line segment. The
2-dimensional zonotopes of rank $n$ are exactly the centrally symmetric
convex $2n$-gons, and every almost cover of such a polygon with lines
requires at least $n$ lines. These are the cases when most vertices of
$C_n$ are lost along the projection. In the other extreme, when the projection
is the identical transformation, $\ac(Z)=\rk(Z)$ follows from
the Alon--F\"uredi theorem, and Theorem \ref{main} asserts the same
for $Z=\Pi_{n-1}$. The construction in Section 4 also exhibits this property.

\begin{conj}
Every almost cover of the vertices of a zonotope $Z$
consists of at least $\rk(Z)$ hyperplanes.
\label{zonotope}
\end{conj}

\noindent
Further examples supporting this conjecture can be found in \cite{K},
where we also investigate the extent to which our method can be applied
to Coxeter permutahedra.
If true, this bound would be best possible in the following sense. Suppose that
the full dimensional zonotope $Z$ is the Minkowski sum of the line segments
$[0,v_i]$, $1\le i\le n$. Denoting by $H$ the affine hull of the points
$v_i$ and assuming $0\not\in H$,
the hyperplanes $H, 2H, \dots, nH$ cover each vertex of $Z$ except $0$.
Since an arbitrary scaling of the generating vectors of a zonotope 
preserves the face lattice of the zonotope,
each zonotope is combinatorially equivalent to a zonotope $Z$ which has
an almost cover with $\rk(Z)$ hyperplanes.
\medskip

\noindent
{\bf Note added.} During the refereeing process Conjecture \ref{zonotope}
was refuted by G\'abor Dam\'asdi \cite{D}.


\end{document}